\documentclass[11pt,a4paper]{article}

\usepackage{amsmath,amsthm,amsopn,amsfonts,amssymb,dsfont,color}
\usepackage[dvips]{graphicx}
\usepackage{psfrag}

\usepackage{a4,a4wide}

\def\ep{{\varepsilon}}
\def\R{\mathbb R}

\newtheorem{theo}{\textbf{Theorem}}[section]

\newtheorem{lem}[theo]{\textbf{Lemma}}

\newtheorem{cor}[theo]{\textbf{Corollary}}

\newtheorem{assumption}[theo]{\textbf{Assumption}}
\newtheorem{rem}[theo]{\textbf{Remark}}

\title{Slowing Allee effect vs. accelerating heavy tails in monostable reaction diffusion equations}
\date{}

\begin{document}

\maketitle
\begin{center}
{\large\bf Matthieu Alfaro \footnote{ I3M, Universit\'e de
Montpellier 2, CC051, Place Eug\`ene Bataillon, 34095 Montpellier
Cedex 5, France. E-mail: matthieu.alfaro@univ-montp2.fr}.}\\
[2ex]

\end{center}




\vspace{10pt}

\begin{abstract} We focus on the spreading properties of solutions of monostable reaction-diffusion equations.
Initial data are assumed to have heavy tails, which tends to accelerate the invasion phenomenon. On the other hand,
 the nonlinearity involves a weak Allee effect, which tends to slow down the process. We study the balance between the two effects. For algebraic tails, we prove the
exact separation between \lq\lq no acceleration and acceleration''. This implies in particular that, for tails exponentially
unbounded but lighter than algebraic, acceleration never occurs in
presence of an Allee effect. This is in sharp contrast with the
KPP situation \cite{Ham-Roq-10}. When algebraic tails lead to acceleration despite the Allee effect, we also give an accurate estimate of the position of the level sets.
\\


\end{abstract}

\maketitle
\section{Introduction} \label{s:intro}

In this paper we are concerned with the {\it spreading properties}
of $u(t,x)$ the solution of the monostable reaction-diffusion
equation
\begin{equation}\label{eq}
\partial _t u=\partial _{xx}u+f(u),\quad t>0,\, x\in \R,
\end{equation}
when the initial data is {\it front-like} and has a {\it heavy
tail}. When the nonlinearity $f$ is of the Fisher-KPP type, Hamel
and Roques \cite{Ham-Roq-10} proved that such solutions spread by
accelerating and precisely estimated the position of the level
sets of $u(t,\cdot)$ as $t\to \infty$, revealing that they are
propagating exponentially fast. The goal of the present paper is
to introduce a  {\it weak Allee effect}, by letting $f'(0)=0$, and
study the balance between such a slowing effect and the
acceleration that heavy tails tend to induce. We prove that, for
data with heavy tails but lighter than algebraic, acceleration is
cancelled by any weak Allee effect (even if very small). For
algebraic tails, we prove the exact separation between \lq\lq
no acceleration and acceleration''. In the latter case, we estimate the
position of the level sets of $u(t,\cdot)$ as $t\to \infty$,
revealing that they are propagating polynomially fast.

\medskip

\noindent{\bf Heavy tails in the Fisher-KPP context.} In some
population dynamics models, a common assumption is that the growth
is only slowed down by the intra-specific competition, so that the
growth per capita is maximal at small densities. This leads to
consider the reaction diffusion equation \eqref{eq} --- where the
quantity $u$ stands for a normalized population density--- with
nonlinearities $f$ of the Fisher-KPP type, namely
$$
f(0)=f(1)=0, \quad\text{ and } \quad 0<f(s)\leq f'(0)s, \quad \forall s\in(0,1).
$$
The simplest example $f(s)=s(1-s)$ was first introduced by Fisher
\cite{Fis-37} and Kolmogorov, Petrovsky and Piskunov
\cite{Kol-Pet-Pis-37} to model the spreading of advantageous genetic
features in a population.

In such situations, it is well known that the way the front like
initial data --- in the sense of Assumption \ref{ass:initial}---
approaches zero at $+\infty$ is of dramatic importance on  the
propagation, that is the invasion for large times of the unstable
steady state $u\equiv 0$ by the stable steady state $u\equiv 1$. To describe
such phenomenon, one can use the notion of spreading speed (if it
exists): for a given front like initial data, we say that
$c=c(u_0)\in\R$ is the spreading speed of the solution $u(t,x)$ of
\eqref{eq} if
$$
\min _{x\leq vt} u(t,x)\to 1 \text{ as  } t\to \infty \text{ if }
v<c, \quad \max _{x\geq vt} u(t,x)\to 0 \text{ as  } t\to \infty
\text{ if } v>c.
$$

For initial data with a exponentially bounded tail (or light tail) at
$+\infty$, there is a spreading speed $c\geq c^*:=2\sqrt{f'(0)}$
which is selected by the rate of decay of the tail. More
precisely, if $u_0(x)=\mathcal O(e^{-\sqrt{f'(0)}x})$ as $x\to
+\infty$ (including left compactly supported initial data) then
$c=c^*=2\sqrt{f'(0)}$, whereas if $u_0(x)$ decays like $e^{-\lambda
x}$, $0<\lambda<\sqrt{f'(0)}$, then $c=\lambda +\frac{f'(0)}\lambda
>c^*$. There is a large literature on such results and improvements.
Let us mention among others the works \cite{Fis-37},
\cite{Kol-Pet-Pis-37}, \cite{McKea-75}, \cite{Had-Rot-75},
\cite{Kam-76}, \cite{Uch-78}, \cite{Aro-Wei-78}, \cite{Kis-Ryz-01},
\cite{Ber-Ham-Nad-05, Ber-Ham-Nad-10}, \cite{Ber-Ham-Nad-08}, \cite{Ham-Sir-10} and the references therein. More recently, the authors in
\cite{Ham-Nad-12} considered the case when the initial data is
trapped between two exponentially decreasing tails, revealing
further properties which enforce to reconsider the notion of
spreading speed.

On the other hand, Hamel and Roques \cite{Ham-Roq-10} recently
considered the case of initial data with heavy tail (or not
exponentially bounded), namely
$$
\lim_{x\to +\infty} u_0(x)e^{\ep x}=0, \quad \forall \ep >0.
$$
Typical examples are algebraic tails but also \lq\lq lighter heavy
tails'', see \eqref{ex1} or \eqref{ex2}, and \lq\lq very heavy tails'', see \eqref{lourd}. In this context, it is then proved in \cite{Ham-Roq-10} that, for
any $\lambda\in(0,1)$, the $\lambda$ level set of $u(t,\cdot)$
travel infinitely fast as $t\to \infty$, thus revealing an
acceleration phenomenon (which in particular prevents the
existence of a spreading speed). Also, the location of these level
sets is estimated in terms of the heavy tail of the initial data.

\medskip

Related results exist for the integro-differential
equation of the KPP type
\begin{equation}\label{eq-garnier}
\partial _t u=J*u-u+f(u),
\end{equation}
where the kernel $J$ allows to take into account rare long-distance dispersal events. Here, the
initial data is typically compactly supported and this is the tail of the dispersion kernel $J$ that determines how
fast is the invasion. If the kernel is exponentially bounded, then propagation occurs at a constant speed, as can be
 seen in \cite{Wei-82}, \cite{Cov-preprint}, \cite{Cov-Dup-07} among others.
 More recently, Garnier \cite{Gar-11}  proved an acceleration phenomenon for kernels which are not
  exponentially bounded, so that \eqref{eq-garnier} is an accurate model to explain the Reid's paradox of rapid plant migration (see \cite{Gar-11} for references on this issue).

To conclude on acceleration phenomena in Fisher-KPP type equations, let us mention the case
when the Laplacian is replaced by the generator of a Feller semigroup, a typical example being
\begin{equation}\label{eq-Cab-Roq-13}
\partial _t u=-(-\partial _{xx})^\alpha u+f(u), \quad 0<\alpha<1,
\end{equation}
where $-(-\partial _{xx})^{\alpha}$ stands for the Fractional Laplacian, whose symbol is $\vert\xi\vert ^{2\alpha}$.
 In this context, it was proved by Cabr\'e and Roquejoffre \cite{Cab-Roq-13} that, for a compactly supported
 initial data, acceleration always occurs, due to the algebraic tails of the Fractional Laplacian.

\medskip

\noindent{\bf Heavy tails vs. Allee effect.} In population dynamics, due for instance to
 the difficulty to find mates or to the lack of genetic diversity at low density, the KPP
  assumption is unrealistic in some situations. In other words, the growth per capita is no longer maximal at small densities, which is referred to as an Allee effect.

 \begin{rem} In the sequel, by Allee effect, we always mean  weak Allee effect. To take
  into account a {\it strong} Allee effect, for which the growth of the population is negative
   at small densities, the common nonlinearity is of the bistable type. In such a framework,
   heavy tails typically do not lead to acceleration  \cite{Bat-Fif-Ren-Wan-97}, \cite{Che-97}, \cite{Gui-Zha-preprint}, \cite{Ach-Kue-preprint}.
\end{rem}

In this Allee effect context, if $f'(0)>0$ the situation --- even if more complicated--- is more or less comparable to the
 KPP situation: most of the above qualitative results remain valid. On the other hand,
 much less is known in the degenerate situation where $f'(0)=0$, for which typical nonlinearities take the form
$$
f(s)=rs^\beta(1-s^\delta), \quad r>0,\, \beta >1,\, \delta >0.
$$

In this work, we focus on the local equation \eqref{eq} with such an
Allee effect. The first question which arises is whether or not propagation (in
the sense of invasion) still occurs for equation \eqref{eq}. It turns out that, in some situations, quenching may occur.
This happens typically when a compactly supported initial data is
too small (in some $L^{1}$ sense) and $\beta>3$. On the other hand, any compactly supported initial data
not too small (in some $L^{1}$ sense) will lead to invasion in the sense that
$$
\lim_{t\to \infty} u(t,x)=1, \quad \text{ locally uniformly in
space}.
$$
Such results were proved by \cite{Xin-93}, \cite{Beb-Li-Li-97}, \cite{Zla-05}. See also earlier works
\cite{Kan-64}, \cite{Roq-97}, \cite{Xin-00} for the case when the
nonlinearity is of the ignition type.

Since we will consider front like initial data (see below
for a precise statement), invasion will always occur. A natural
question is therefore to study the balance between the Allee
effect (whose strength is measured by $\beta >1$) which tends to
slow down the invasion process, and heavy tails which tend to
accelerate. Let us mention some numerical results
\cite{She-Mar-96}, \cite{Kay-She-Leo-01} for the nonlinearity
$f(s)=s^{2}(1-s)$. Also, for nonlinearities $f(s)=s^\beta (1-s)$
and algebraic initial data, matched asymptotic expansions
\cite{Nee-Bar-99}, \cite{Lea-Nee-Kay-02} have been used to determine if
the solution travels with finite or infinite speed.

In this work, we provide a rigorous description of the competition between the Allee
effect and the heavy tail for equation \eqref{eq}. For algebraic tails, we prove the
separation between acceleration or not (depending on the strength of the Allee effect).
Also, when acceleration occurs, we precisely estimate the location of the level sets of the solution.
This separation for algebraic tails immediately implies that acceleration never occurs for lighter tails
 (even if the Allee effect is very small), and always occurs for heavier tails (even if the Allee effect is very large). This is in sharp contrast with the KPP situation \cite{Ham-Roq-10}.

\medskip
As far as the integro-differential equation \eqref{eq-garnier} with an
Allee effect is concerned, the question of propagation or not has been recently
studied by \cite{Zha-11}, \cite{Zha-Li-Wan-12}. One may then wonder what is the exact balance between the Allee effect and dispersion kernels with heavy tails. We refer to
\cite{Alf-Cov-preprint} for first results in this direction.

Last, notice that the question of acceleration or not in the nonlocal equation \eqref{eq-Cab-Roq-13} with
 an Allee effect has been recently solved by  Gui and Huan \cite{Gui-Hua-preprint}.
 Considering $\partial _t u=-(-\partial _{xx})^\alpha u+u^\beta (1-u)$, they show that:
 for $0<\alpha\leq 1/2$ acceleration always occurs whatever $\beta >1$ by comparing with
  an ignition type problem; next, for $1/2<\alpha<1$, acceleration occurs if and only if $\beta <\frac{2\alpha}{2\alpha -1}$. See \cite{Gui-Hua-preprint} for more precise results.

\section{Assumptions and main results}\label{s:results}

Through this work, and even if not recalled, we always assume the following
 on the initial condition. Notice that, in each result, we clearly state the heavy tail assumption which is therefore not included below.

\begin{assumption}(Initial condition)\label{ass:initial} The initial condition $u_0:\R\to \left[0,1\right]$ is uniformly continuous and asymptotically front-like, in the sense that
\begin{equation}
\label{front-like}
u_0>0 \; \text{ in } \R,\quad \liminf _{x\to -\infty} u_0(x)>0,\quad \lim _{x\to +\infty} u_0(x)=0.
\end{equation}
\end{assumption}

Even if not recalled, we always assume the following on the nonlinearity $f$. Notice
 that, in  each result, we clearly quantify the degeneracy assumption (Allee effect) which is therefore not included below.

\begin{assumption}(Degenerate monostable nonlinearity)\label{ass:f} The nonlinearity $f:\left[0,1 \right]\to \R$ is of the class $C^{1}$, and is of the monostable type, in the sense that
$$
f(0)=f(1)=0, \quad f>0 \; \text{ in } (0,1).
$$
The steady state 0 is degenerate, that is $f'(0)=0$.
\end{assumption}

The simplest example of such a degenerate monostable nonlinearity is
 given by $f(s)=s^\beta(1-s)$, with $\beta >1$ (in contrast with the KPP situation $\beta =1$).

In the sequel, we always denote by $u(t,x)$ the solution of \eqref{eq} with initial condition $u_0$. From the above assumptions and the comparison principle, we immediately get
$$
0<u(t,x)<1 \quad \forall (t,x)\in(0,\infty)\times\R.
$$
Also, as announced in the introduction, the state $u\equiv 1$ does invade the whole line
 $\R$ as $t\to \infty$. Indeed, define $\eta:=\inf _{x\leq 0} u_0(x)>0$. In view of \cite[Theorem 1.1]{Zla-05},
 the solution $v(t,x)$ of \eqref{eq} with initial data $v_0(x)=\eta \chi _{(-\infty,0)}(x)$ satisfies $\lim _{t\to \infty} \inf _{x\leq \gamma t} v(t,x)=1$
 for some $\gamma >0$. From $v_0\leq u_0$ and the comparison principle,  the same holds true for $u(t,x)$:
\begin{equation}
\label{invasion}
\lim _{t\to \infty} \inf _{x\leq \gamma t} u(t,x)=1,
\end{equation}
so that propagation is at least linear. Notice also that the proof of \cite[Theorem 1.1 part a)]{Ham-Roq-10} does not require
 the KPP assumption and can then be reproduced to get
\begin{equation}
\label{zero-a-droite}
\lim _{x\to+\infty}u(t,x)=0,\quad \forall t\geq 0.
\end{equation}

In order to state our results we define, for any $\lambda \in (0,1)$ and $t\geq 0$,
$$
E_\lambda(t):=\{x\in\R:\, u(t,x)=\lambda\}
$$
the $\lambda$ level set of $u(t,\cdot)$. In view of \eqref{invasion} and \eqref{zero-a-droite}, for any $\lambda\in(0,1)$, there is a time $t_\lambda >0$ such that
\begin{equation}
\label{nonvide}
\emptyset \neq E_\lambda(t) \subset (\gamma t, +\infty), \quad \forall t\geq t_\lambda.
\end{equation}

Our first main result states that, for algebraic initial tail, acceleration can be blocked by a strong enough Allee effect.

\begin{theo}(Cancelling acceleration by Allee effect)\label{th:no-acc} Let $\alpha >0$ and $\beta >1$ be such that
\begin{equation}\label{alpha-beta-no-acc}
 \beta \geq 1+\frac 1 \alpha.
\end{equation}
Assume that there are $C>0$ and $x_0>1$ such that
\begin{equation}
\label{algebraic-no-acc}
u_0(x)\leq \frac{C}{x^\alpha},\quad \forall x\geq x_0.
\end{equation}
Assume that there are $r>0$, $\delta >0$ and $s_0\in(0,1)$ such that
\begin{equation}
\label{nonlinearity-no-acc}
f(s)\leq r s^\beta(1-s^\delta),\quad \forall 0\leq s\leq s_0.
\end{equation}
Then, there is a speed $c>0$ such that, for any $\lambda \in(0,1)$,
there is a time $T_\lambda \geq t_\lambda$ such that
\begin{equation}
\label{no-acc}
\emptyset \neq E_\lambda(t) \subset (\gamma t, ct), \quad \forall t\geq T_\lambda.
\end{equation}
\end{theo}

On the one hand, for any  Allee effect $\beta >1$, one can find some
initial conditions with algebraic tail (whose power is large enough) so that the solutions do not
accelerate, as can be seen from \eqref{no-acc}.  On the other hand, for
any initial condition with algebraic tail, one can find some  Allee effect (strong enough) so that acceleration is
 cancelled. This is in sharp contrast with the KPP situation $\beta =1$ studied in \cite{Ham-Roq-10}.

Another difference with \cite{Ham-Roq-10} is concerned with heavy
tails that are lighter than algebraic ones, for which acceleration
is always cancelled whatever the strength of the  Allee effect.

\begin{cor}(Heavy tails lighter than algebraic)\label{cor:expracine} Let $\beta>1$ be arbitrary. Assume that for all $\alpha >0$, there are $C_\alpha >0$ and $x_0 ^\alpha>1$ such that
\begin{equation}\label{hyp-queues-exp-lente}
u_0(x)\leq \frac{C_\alpha}{x^{\alpha}}, \quad \forall
x\geq x_0 ^\alpha.
\end{equation}
Assume \eqref{nonlinearity-no-acc}. Then, for any $\lambda \in(0,1)$, the no acceleration result \eqref{no-acc} holds.
\end{cor}

The above result is independent on $\beta >1$ and is valid, among others, for initial data satisfying
\begin{equation}
\label{ex1}
u_0(x)\leq C e^{-a x /(\ln  x )}, \quad \forall
x\geq 2, \quad \text{ for some } C>0, a>0,
\end{equation}
or
\begin{equation}
\label{ex2}
u_0(x)\leq C e^{-a x^{b}}, \quad \forall x\geq 1, \quad
\text{ for some } C>0, a>0,\, 0<b<1.
\end{equation}
For such tails, any  Allee effect cancels the acceleration, whereas in the
KPP case acceleration always occurs \cite{Ham-Roq-10}. The proof of Corollary \ref{cor:expracine} is obvious: for a given $\beta
>1$, select a large $\alpha>0$ such that \eqref{alpha-beta-no-acc} holds,
and then combine \eqref{hyp-queues-exp-lente} with Theorem
\ref{th:no-acc}.

From Theorem \ref{th:no-acc} and Corollary \ref{cor:expracine},
\lq\lq the transition from no acceleration to acceleration'' seems
to take place for algebraic tails. This is confirmed by our next
main result, which is concerned with the case when the Allee
effect is not strong enough to prevent the acceleration induced by
algebraic tails.

\begin{theo}(Acceleration despite  Allee effect)\label{th:acc} Let $\alpha >0$ and $\beta >1$ be such that
\begin{equation}\label{alpha-beta-acc}
 \beta < 1+\frac 1 \alpha.
\end{equation}
Assume that there are $C>0$ and $x_0>1$ such that
\begin{equation}
\label{algebraic-acc}
u_0(x)\geq \frac{C}{x^\alpha},\quad \forall x\geq x_0.
\end{equation}
Assume that there are $r>0$, $\delta >0$ and $s_0\in(0,1)$ such that
\begin{equation}
\label{nonlinearity-acc}
f(s)\geq r s^\beta(1-s^\delta),\quad \forall 0\leq s\leq s_0.
\end{equation}
Then, for any $\lambda \in(0,1)$, any small $\ep>0$,  there is a
time $T_{\lambda,\ep}\geq t_\lambda$ such that
\begin{equation}
\label{lowerbound} E_\lambda(t)\subset (x^-(t),+\infty)\quad
\forall t\geq T_{\lambda,\ep}, \quad x^-(t):=\left((r-\ep)C^{\beta
-1} (\beta -1)t\right)^{\frac{1}{\alpha(\beta-1)}}.
\end{equation}
\end{theo}

 Combining Theorem \ref{th:no-acc} and Theorem
\ref{th:acc}, we get a complete picture of the propagation
phenomenon. Indeed, in the $(\alpha,\beta)$ plane there is no
acceleration above or on the hyperbola $\beta=1+\frac 1 \alpha$.
On the other hand, strictly below the
 hyperbola acceleration occurs (see Figure \ref{hyperbola}).

\begin{figure}[!h]
\begin{center}
\includegraphics[width=0.5\linewidth]{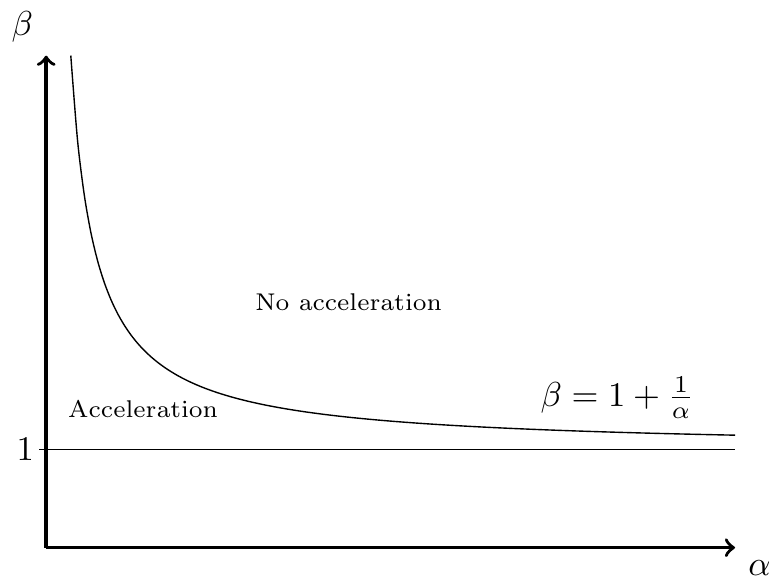}
\caption{$\beta$-Allee effect vs. $\alpha$-algebraic tail.} \label{hyperbola}
\end{center}
\end{figure}

As an immediate corollary of Theorem \ref{th:acc} we get that, for tails heavier than algebraic ones,
acceleration always occurs whatever the strength $\beta >1$ of the Allee effect. Typical examples of such tails are
\begin{equation}\label{lourd}
u_0(x)\geq \frac C{(\ln x)^{b}}, \quad \forall
x\geq 2, \quad \text{ for some } C>0, b>0.
\end{equation}

Our last result consists in providing upper bounds on the level sets of $u(t,x)$, when the algebraic tail is
 stronger than the Allee effect so that acceleration occurs. Combining with  the lower bounds  of Theorem \ref{th:acc}, this yields an accurate \lq\lq sandwich'' of the level sets.

\begin{theo}(Sandwich of the accelerating level sets)\label{th:sandwich} Let $\alpha >0$, $\delta >0$ and $\beta >1$ be such that \eqref{alpha-beta-acc} holds.
Assume that there are $C>0$, $\overline C>0$ and $x_0>1$ such that
\begin{equation}
\label{algebraic-accbis}
\frac{\overline C}{x^{\alpha}}\geq u_0(x)\geq \frac{C}{x^\alpha},\quad \forall x\geq x_0.
\end{equation}
Assume \eqref{nonlinearity-acc} and that there is $\overline r>0$ such that
\begin{equation}
\label{nonlinearity-accbis}
\overline r s^\beta\geq f(s),\quad \forall 0\leq s\leq 1.
\end{equation}
Then, for any $\lambda \in(0,1)$, any small $\ep>0$,  there is a time $\overline T_{\lambda,\ep}\geq t_\lambda$ such that
\begin{equation}
\label{levelset}
E_\lambda(t)\subset (x^-(t),x^{+}(t)),\quad \forall t\geq \overline T _{\lambda,\ep},
\end{equation}
where
$$
 x^-(t):=\left((r-\ep)C^{\beta -1}(\beta -1)t\right)^{\frac{1}{\alpha(\beta-1)}}, \quad x^+(t):=\left((\overline r+\ep)\overline C ^{\beta -1}(\beta -1)t\right)^{\frac{1}{\alpha(\beta-1)}}.
 $$
\end{theo}

\medskip

The organization of the paper is as follows. In Section \ref{s:no-acceleration} we consider the situation where
 the Allee effect is stronger than the algebraic tail so that acceleration does not occur, that is we prove
  Theorem \ref{th:no-acc}. In Section \ref{s:acceleration} we consider the opposite situation, proving the acceleration as stated in
  Theorem \ref{th:acc}. Last, in Section \ref{s:upper}, we prove the upper estimates on the spreading of the level sets when accelerating, thus completing the proof of Theorem \ref{th:sandwich}.

\section{Cancelling acceleration by Allee effect }\label{s:no-acceleration}

In this short section, we prove Theorem \ref{th:no-acc}. The formal argument is very simple,
close to the rigorous proof and enlightening. We therefore take the liberty to present it: in order to
 block acceleration, we aim at finding a speed $c>0$ and a power $p>0$ such that $w(z):=\frac 1{z^p}$ is a supersolution of the associated traveling wave equation for $z>>1$, that is
$$
w''(z)+cw'(z)+f(w(z))\leq 0.
$$
In view of \eqref{nonlinearity-no-acc} this is enough to have
$$
\frac{p(p+1)}{z^{p+2}}-\frac{cp}{z^{p+1}}+\frac{r}{z^{p\beta}}\leq 0,
$$
for large $z>>1$, which requires $p+1\leq p\beta$, that is $\frac{1}{\beta-1}\leq p$. On the other
 hand we also need the ordering at initial time, which in view of \eqref{algebraic-no-acc}, requires
 $p\leq \alpha$. Hence one needs $\frac 1{\beta -1}\leq \alpha$, so that the hyperbola separation \eqref{alpha-beta-no-acc} arises very naturally. Let us now make this formal argument precise.

\medskip

We define
$$
p:=\frac{1}{\beta-1},\quad w(z):=\frac{K}{z^p} \quad \text{ for } z\geq z_0:=K^{1/p},
$$
where $K>1$.

\begin{lem}(Supersolutions traveling at constant speed)\label{lem:sursoltw} Let assumptions of Theorem \ref{th:no-acc} hold. Then, for any $K>1$, there is $c>0$ such that
$$
w''(z)+cw'(z)+f(w(z))\leq 0, \quad \forall z\geq z_0.
$$
\end{lem}

\begin{proof} In view of \eqref{nonlinearity-no-acc}, if $z\geq z_1:=\left(\frac{K}{s_0}\right)^{1/p}>z_0$ then $w(z)=\frac{K}{z^p}\leq s_0$ so that
$$
w''(z)+cw'(z)+f(w(z))\leq \frac{Kp(p+1)}{z^{p+2}}-\frac{cKp}{z^{p+1}}+\frac{rK^\beta}{z^{p\beta}}=\frac{Kp(p+1)}{z^{p+2}}-\frac{cKp-rK^{\beta}}{z^{p+1}}.
$$
Choosing $c>r(\beta -1)K^{\beta-1}$, the above is clearly negative
for $z$ large enough, say $z\geq z_2$. Last, on the remaining
compact region $z_0\leq z\leq z_2$, we have
\begin{eqnarray*}
w''(z)+cw'(z)+f(w(z))&=&\frac{Kp(p+1)}{z^{p+2}}-\frac{cKp}{z^{p+1}}+f(w(z))\\
&\leq& \frac{Kp(p+1)}{z_0 ^{p+2}}-\frac{cKp}{z_2^{p+1}}+\Vert f\Vert _{L^{\infty}(0,1)}\\
&\leq & 0
\end{eqnarray*}
by enlarging $c$ if necessary.
\end{proof}

We can now complete the proof of Theorem \ref{th:no-acc}. We select $K=\max(1,C)$, where $C>0$ is
  the constant that appears in \eqref{algebraic-no-acc}, and $c>0$ the associated speed given by the above lemma. Then
$$
v(t,x):=\min\left(1,w(x-x_0+1-ct)\right),
$$
is a supersolution of equation \eqref{eq}. Indeed, since 1 solves \eqref{eq} it is enough
 to deal with the region where $v(t,x)<1$, that is $z:=x-x_0+1-ct>z_0$, where it directly
  follows from the above lemma since $(\partial _t v-\partial _{xx}v-f(v))(t,x)=(-cw'-w''-f(w))(z)$.  Also we have
$$
v(0,x)=\min\left(1,\frac{K}{(x-x_0+1)^p}\right)\geq u_0(x),
$$
in view of $u_0\leq 1$, the assumption on the tail \eqref{algebraic-no-acc}, $K\geq C$ and $p=\frac{1}{\beta-1}\leq \alpha$. It follows from the comparison principle that
$$
u(t,x)\leq v(t,x)=\min\left(1,w(x-x_0+1-ct)\right).
$$
Now, let $\lambda\in(0,1)$ be given. In view of \eqref{nonvide}, for $t\geq t_\lambda$,  we can pick $x\in E_\lambda(t)$, and the above inequality enforces
$$
x\leq x_0-1+\left(\frac{K}{\lambda}\right)^{\beta-1}+ct\leq (c+1)t,
$$
for all $t\geq T_\lambda$, if $T_\lambda \geq t_\lambda$ is sufficiently large. This proves the
 upper bound in \eqref{no-acc}. The lower bound in \eqref{no-acc} being known since \eqref{nonvide}, this completes the proof of Theorem \ref{th:no-acc}. \qed

\section{Acceleration despite  Allee effect}\label{s:acceleration}

In this section, we analyze the situations where the algebraic tail is stronger than
 the Allee effect, in the sense of \eqref{alpha-beta-acc}, so that the solution accelerates. Namely,
 we prove Theorem \ref{th:acc}. Notice that, in view of \eqref{algebraic-acc} and the comparison principle, we only need to consider the case where
\begin{equation}
\label{algebraic-acc2}
u_0(x)= \frac{C}{x^\alpha},\quad \forall x\geq x_0.
\end{equation}

\subsection{An accelerating small bump as a subsolution}\label{ss:bump}

The main difficulty is to construct a subsolution which has the form of a small bump and travels to the
 right by accelerating. To do so in a KPP situation, the authors in \cite{Ham-Roq-10} consider a perturbation of the solution of $\frac{d}{dt} w(t,x)= r w(t,x)$ with
$w(0,x)=u_0(x)$ as initial data, where $x\in\R$ serves as a parameter. Guided by this approach, we shall rely --- in our degenerate situation---  on the
solution of $\frac{d}{dt} w(t,x)= r w^\beta(t,x)$ with
$w(0,x)=u_0(x)$ as initial data, where $x\in\R$ serves as a parameter. Computations are more involved, and it will turn out that
the higher order term of the nonlinearity --- typically of the form $f(s)=rs^\beta(1-s^\delta)$--- will play a role, so that we first need to assume \eqref{enplus}. We start with some preparations.

\medskip

Let $\ep>0$ small be given. We first make the additional assumption (to be removed in the end of the section)
\begin{equation}
\label{enplus}
\beta<1+\delta.
\end{equation}
We can therefore  select a $\rho>0$ such that
\begin{equation}
\label{def-rho}
\max\left(\frac{\beta r}{1+\delta},r-\ep\right)<\rho <r.
\end{equation}
Then define
\begin{equation}\label{def-w}
w(t,x):=\frac{1}{\left(\frac{1}{u_0^{\beta-1}(x)}-\rho (\beta -1)t\right)^{\frac{1}{\beta -1}}}\quad \text{ for } 0\leq t <T(x):=\frac{1}{\rho(\beta-1)u_0^{\beta-1}(x)},
\end{equation}
which solves
\begin{equation}
\label{edo-w}
\partial _t w(t,x)=\rho w^\beta (t,x),\quad w(0,x)=u_0(x).
\end{equation}
\begin{rem}
Notice that, as $x\to+\infty$ the interval of existence $(0,T(x))$ of the solution $w(t,x)$ becomes large since, in view of \eqref{algebraic-acc2},
$$
T(x)= \frac{x^{\alpha(\beta -1)}}{\rho(\beta-1)C^{\beta -1}}, \quad \forall x\geq x_0.
$$
Also, since $\alpha (\beta-1)<1$, we will have \lq\lq enough place'' to observe the
 acceleration phenomenon which, in some sense, is given by $x(t)\sim \mathcal O(t^{\frac{1}{\alpha(\beta-1)}})$ as $t\to \infty$, as can be seen in Theorem \ref{th:sandwich}.
\end{rem}
Straightforward computations yield
\begin{equation}
\label{calcullaplacien}
\partial _{xx} w(t,x)=g(x)w^\beta(t,x)+\beta h(x)w^{2\beta-1}(t,x),\end{equation}
 where
\begin{equation}\label{defgh}
g(x):=\frac{u_0''(x)}{u_0^\beta(x)}-\beta\frac{(u_0'(x))^2}{u_0^{\beta +1}(x)}, \quad h(x):=\frac{(u_0'(x))^2}{u_0  ^{2\beta}(x)}.
\end{equation}
In view of \eqref{algebraic-acc2} and $\beta<1+\frac 1 \alpha$, we see that both $g(x)$ and $h(x)$ tend to zero as $x\to +\infty$. Let us therefore select $x_1>x_0$ such that
\begin{equation}
\label{loin1}
\vert g(x)\vert +\beta \vert h(x)\vert \leq \frac{r-\rho}{2} \quad \text{ and }\quad
\vert g(x)\vert +(\delta+\beta) \vert h(x)\vert \leq \frac{\rho-\frac{r\beta}{1+\delta}}{2},\quad \forall x\geq x_1.
\end{equation}
Now, Assumption \ref{ass:initial} implies that
\begin{equation}
\label{def-kappa}
\kappa:=\inf _{x\in(-\infty,x_1)} u_0(x)\in(0,1] .
\end{equation}
Last, we select $A>0$ large enough so that
\begin{equation}
\label{def-A}
A>\max\left(\frac{1}{\kappa ^{\delta}}, \frac{2r}{1+\delta}\left(\rho-\frac{r\beta}{1+\delta}\right)^{-1}\right),
\end{equation}
and
\begin{equation}
\label{def-A2}
\frac{\delta}{1+\delta}\frac{1}{(A(1+\delta))^{1/\delta}}\leq s_0,
\end{equation}
where $s_0$ is as in \eqref{nonlinearity-acc}.
Equipped with the above material, we are now in the position to construct the desired subsolution.

\begin{lem}(An accelerating subsolution)
\label{lem:sub} Let assumptions of Theorem \ref{th:acc} hold. Then the function
$$
v(t,x):=\max\left(0, w(t,x)-Aw^{1+\delta}(t,x)\right)
$$
is a subsolution of equation \eqref{eq} in $(0,\infty)\times \R$.
\end{lem}

\begin{proof} Since 0 solves \eqref{eq} it is enough to consider the  $(t,x)$ for which $v(t,x)>0$. We therefore need to show
$$
\mathcal L v(t,x):=\partial _t v(t,x)-\partial _{xx}v(t,x)-f(v(t,x))\leq 0\quad \text{ when }\, v(t,x)=w(t,x)-Aw^{1+\delta}(t,x)>0.
$$
This implies in particular that $w(t,x)<1/A^{1/\delta}<\kappa$
so that $u_0(x)=w(0,x)\leq w(t,x)<\kappa$ since $t\mapsto w(t,x)$ is increasing. In view of the
 definition of $\kappa$ in \eqref{def-kappa}, this enforces $x\geq x_1$. As a result estimates \eqref{loin1} are
available. On the other hand $v(t,x)\leq \max_{0\leq w\leq
A^{1/\delta}}w-Aw^{1+\delta}=\frac{\delta}{1+\delta}\frac{1}{(A(1+\delta))^{1/\delta}}\leq
s_0$ by \eqref{def-A2}. Hence, it follows from \eqref{nonlinearity-acc} that
\begin{eqnarray*}
f(v(t,x))&\geq& rv^{\beta}(t,x)(1-v^{\delta}(t,x))\\
&\geq& rv^\beta(t,x)-rw^{\beta+\delta}(t,x)\\
&=& rw^{\beta}(t,x)(1-Aw^{\delta}(t,x))^{\beta}-rw^{\beta+\delta}(t,x).
\end{eqnarray*}
Then the convexity inequality $(1-Aw^\delta)^\beta\geq 1-A\beta w^\delta$ yields
$$
f(v(t,x))\geq rw^{\beta}(t,x)-rA\beta w^{\beta+\delta}(t,x)-rw^{\beta+\delta}(t,x).
$$
Using this, \eqref{edo-w}, \eqref{calcullaplacien}, computing $\partial _t w^{1+\delta}(t,x)$ and $\partial _{xx} w^{1+\delta}(t,x)$, we arrive at
\begin{align*}
&\mathcal L v(t,x)\leq\rho w^{\beta}(t,x)-g(x)w^{\beta}(t,x)-\beta h(x)w^{2\beta-1}(t,x)-A\rho(1+\delta)w^{\beta+\delta}(t,x)\\
&\quad \quad \quad \quad +A(1+\delta)g(x)w^{\beta+\delta}(t,x)+A(1+\delta)(\delta+\beta)h(x)w^{2\beta+\delta-1}(t,x)\\
&\quad \quad \quad\quad -rw^{\beta}(t,x)+rA\beta w^{\beta+\delta}(t,x)+rw^{\beta+\delta}(t,x)\nonumber.
\end{align*}
Since $0\leq w<1$ we have $w^{2\beta-1}\leq w^\beta$ and $w^{2\beta+\delta -1}\leq w^{\beta +\delta}$, so that
\begin{align*}
&\mathcal L v(t,x)\leq w^{\beta}(t,x)\left[\rho-r+\vert g(x)\vert +\beta\vert h(x)\vert\right]\\
&\quad \quad \quad\quad +w^{\beta+\delta}(x)\left[A(1+\delta)\left(-\rho+\frac{r\beta}{1+\delta}+\vert g(x)\vert+(\delta+\beta)\vert h(x)\vert\right)+r\right].
\end{align*}
The first inequality in \eqref{loin1} implies $\rho -r+\vert g(x)\vert +\beta\vert h(x)\vert\leq \frac{\rho -r}{2}\leq 0$ and, using the second inequality in \eqref{loin1}, we get
$$
\mathcal  L v(t,x)\leq w^{\beta+\delta}(x)\left[A(1+\delta)\frac{-\rho+\frac{r\beta}{1+\delta}}{2}
 +r\right]\leq 0,
 $$
 thanks to \eqref{def-A}. Lemma \ref{lem:sub} is proved.
\end{proof}

Since $v(0,x)=\max(0,u_0(x)-Au_0^{1+\delta}(x))\leq u_0(x)$, we deduce from the comparison principle that
\begin{equation}
\label{comparison}
u(t,x)\geq v(t,x)=\max(0,w(t,x)-Aw^{1+\delta}(t,x)),\quad \forall (t,x)\in [0,\infty)\times \R.
\end{equation}

\subsection{Lower bounds on the level sets}\label{ss:lower}

{\bf Proof of \eqref{lowerbound} for small $\lambda$, under assumption \eqref{enplus}.} Equipped with the
 above subsolution, whose role is to \lq\lq lift'' the solution $u(t,x)$ on intervals that enlarge with
 acceleration,  we first prove the lower bound \eqref{lowerbound} on the level sets $E_\lambda(t)$ when $\lambda$ is small.

Let us fix
$$
0<\theta<1/A^{1/\delta}.
$$
We  claim that, for any $t\geq 0$, there is a unique $y_\theta (t)\in \R$ such that
$w(t,y_\theta(t))=\theta$, and moreover $y_\theta(t)$ is given by
\begin{equation}
\label{def-ytheta}
y_\theta(t):=\left((\frac C\theta)^{\beta-1}+\rho C^{\beta -1}(\beta -1)t\right)^{\frac{1}{\alpha(\beta-1)}}.
\end{equation}
Indeed, since $\theta<1/A^{1/\delta}<\kappa=\inf _{x\in(-\infty,x_1)}u_0(x)$ and since $w(t,x)\geq w(0,x)=u_0(x)$,
for $w(t,y)=\theta$ to hold one needs $y\geq x_1$. But, when $y\geq x_1>x_0$, one can use formula \eqref{algebraic-acc2}
and then solve equation $w(t,y)=\theta$, thanks to expression \eqref{def-w}, to find the unique solution \eqref{def-ytheta}.

Let us now define the open set
$$
\Omega :=\{(t,x), t>0, x<y_\theta (t)\}.
$$
Let us evaluate $u(t,x)$ on the boundary $\partial \Omega$.
For $t>0$, it follows from \eqref{comparison} that
$$
u(t,y_\theta(t))\geq w(t,y_\theta (t))-Aw^{1+\delta}(t,y_\theta (t))=
\theta -A\theta ^{1+\delta}>0.
$$
On the other hand, for $t=0$ and $x\leq y_\theta (0)=(\frac C \theta)^{1/\alpha}$, we have
$$
u(0,x)\geq \inf _{x\leq (\frac C \theta)^{1/\alpha}}u_0(x)>0,
$$
in view of Assumption \ref{ass:initial}. As a result $\Theta:=\inf _{(t,x)\in \partial\Omega}u(t,x)>0$. Since $\Theta >0$
is a subsolution for equation \eqref{eq}, it follows from the comparison principle that
\begin{equation}
\label{etoile}
u(t,x)\geq \Theta, \quad \forall t\geq 0, \forall x\leq y_\theta(t).
\end{equation}
This implies in particular that, for any $0<\lambda <\Theta$, we have, for all $t\geq t_\lambda$,
\begin{equation}
\label{small-levelset} \emptyset \neq E_\lambda (t)\subset
(y_\theta(t),+\infty)\subset (x^{-}_\rho (t),+\infty),\quad
 x^{-}_\rho (t):=\left(\rho C^{\beta -1}(\beta -1)t\right)^{\frac{1}{\alpha(\beta-1)}},
 \end{equation}
which implies the lower bound \eqref{lowerbound} since $\rho
>r-\ep$. \qed

 \medskip

\noindent {\bf Proof of \eqref{lowerbound} for any $\lambda\in(0,1)$, under assumption \eqref{enplus}.}
 Let us now turn to the the case where $\lambda$ is larger than $\Theta$. Let
 $\Theta\leq \lambda <1$ be given. Let us denote by $v(t,x)$ the solution of \eqref{eq} with initial data
 \begin{equation}\label{initial-data-v}
v_0(x):= \begin{cases} \Theta &\text{
if } x\leq -1
\\
-\Theta x &\text{ if } -1<x<0\\
0 &\text{ if  }x\geq 0.
\end{cases}
\end{equation}
It follows from \cite[Theorem 1.1]{Zla-05} that $
\lim _{t\to\infty}\inf_{x\leq \gamma _1 t} v(t,x)=1$,
for some $\gamma _1 >0$. In particular there is a time $\tau_{\lambda,\ep}>0$ (this time depends on
$\theta$ and therefore on $\ep$ from the above construction of the small bump subsolution) such that
\begin{equation}
\label{v-grand}
 v(\tau_{\lambda,\ep},x)>\lambda, \quad \forall x\leq 0.
 \end{equation}

 On the other hand, it follows from \eqref{etoile} and the definition \eqref{initial-data-v}
 that
 $$
 u(T,x)\geq v_0(x-y_\theta(T)),\quad \forall T\geq 0, \forall x\in \R,
 $$
 so that the comparison principle yields
 $$
 u(T+\tau,x)\geq v(\tau,x-y_\theta(T)),\quad \forall T\geq 0, \forall \tau \geq 0, \forall x\in \R.
 $$
 In view of \eqref{v-grand}, this implies that
 $$
 u(T+\tau_{\lambda,\ep},x)>\lambda,\quad \forall T\geq 0, \forall x\leq y_\theta (T).
 $$
 Hence, for any $t\geq T^{1}_{\lambda,\ep}:=\max (\tau_{\lambda,\ep},t_\lambda)$, if we pick
  a $x\in E_\lambda(t)$ then the above implies $x>y_\theta(t-\tau_{\lambda,\ep})$, that is
 $$
 x>\left((\frac C\theta)^{\beta-1}-\rho C^{\beta-1}(\beta-1)\tau _{\lambda,\ep}+\rho C^{\beta -1}(\beta -1)t\right)^{\frac{1}{\alpha(\beta-1)}}\geq \left((r-\ep)C^{\beta-1}(\beta-1)t\right)^{\frac{1}{\alpha(\beta-1)}},
 $$
for all  $t\geq T_{\lambda,\ep}$, with $T_{\lambda,\ep}>0$  sufficiently large (recall that $\rho>r-\ep$).
This proves the lower bound \eqref{lowerbound} when $\Theta\leq \lambda<1$ and concludes the proof of Theorem \ref{th:acc}, under the additional assumption \eqref{enplus}. \qed

\medskip

\noindent {\bf Relaxing the additional assumption \eqref{enplus}.} When $\beta<1+\delta$ does not hold, let us pick $\delta ^{*}>0$ such that $\beta <1+\delta ^{*}$ and define $r^{*}:=r-\frac \ep 2$. It follows from \eqref{nonlinearity-acc} that there is $s_0^{*}\in(0,1)$ such that
$$
f(s)\geq r^{*}s^{\beta}(1-s^{\delta ^{*}}), \quad \forall 0\leq s\leq s_0^{*}.
$$
Hence, from the above analysis, \eqref{lowerbound} is available
with $r^{*}$ in place of $r$. This concludes the proof of Theorem
\ref{th:acc}. \qed

\section{Upper bounds on the level sets when acceleration}\label{s:upper}

In this section, we sandwich the level sets $E_\lambda(t)$ of the
solution $u(t,x)$ when acceleration occurs, namely we prove
Theorem \ref{th:sandwich}. In view of Theorem \ref{th:acc}, it
only remains to prove the upper estimate in \eqref{levelset}.

\medskip

Let $\lambda \in(0,1)$ and  $\ep>0$ small be given.
Up to enlarging $x_0>1$ which appears in \eqref{algebraic-accbis}, we can assume without loss of generality that
\begin{equation}
\label{xzerogrand}
\frac{\alpha(\alpha+1+2\beta\alpha)}{x_0^2}\leq\frac \ep 2,
\end{equation}
and $\frac{\overline C}{x_0^{\alpha}}<1$. Then, up to enlarging $\overline C>0$ which also appears in \eqref{algebraic-accbis}, we can assume without loss of generality that
\begin{equation}
\label{Cbargrand}
\frac{\overline C}{x_0^{\alpha}}=1.
\end{equation}
Now, for these $x_0>1$ and $\overline C>0$, in view of \eqref{algebraic-accbis} and the comparison principle, it is enough to prove the upper bound in \eqref{levelset} when
\begin{equation}
\label{algebraic-acc3}
u_0(x)= \frac{\overline C}{x^\alpha},\quad \forall x\geq x_0.
\end{equation}

Let us select
$$
\rho:=\overline r+\frac \ep 2.
$$
We then define
$$
\psi(t,x):=\min\left(1, w (t,x):=\frac{1}{\left(\frac{1}{u_0^{\beta-1}(x)}-\rho (\beta -1)t\right)^{\frac{1}{\beta -1}}}\right),
$$
where $w(t,x)$ is as in \eqref{def-w}.
We claim that $\psi$ is a supersolution for equation \eqref{eq} in the domain $\Omega:=(0,\infty)\times(x_0,+\infty)$.
Indeed, since $1$ solves \eqref{eq}, it suffices to consider the points $(t,x)$ where $\psi(t,x)=w(t,x)<1$. In view of
$$
\partial _t w(t,x)=\rho w ^\beta (t,x)=(\overline r+\frac \ep 2)w ^\beta(t,x),
$$
and inequality \eqref{nonlinearity-accbis}, some straightforward computations yield
\begin{align}
 &\partial _t w(t,x)-\partial _{xx}w(t,x)-f(w(t,x))\nonumber\\
&\quad \quad \quad\quad\geq \frac \ep 2 w ^\beta(t,x)-g(x)w^\beta (t,x)-\beta h(x) w ^{2\beta -1}(t,x)\nonumber\\
&\quad \quad \quad\quad\geq w ^\beta(t,x)\left( \frac{\ep}{2}-\vert g(x)\vert -\beta \vert h(x)\vert \right)\label{qqch},
\end{align}
since $0<w(t,x)<1$, and where $g(x)$ and $h(x)$ were defined in \eqref{defgh}, that is
\begin{equation*}\label{defgh2}
g(x):=\frac{u_0''(x)}{u_0^\beta(x)}-\beta\frac{(u_0'(x))^2}{u_0^{\beta +1}(x)}, \quad h(x):=\frac{(u_0'(x))^2}{u_0  ^{2\beta}(x)}.
\end{equation*}
In view of expression \eqref{algebraic-acc3}, some straightforward
computations yield that, for any $x\geq x_0$,
\begin{align*}
 &\vert g(x)\vert +\beta \vert h(x)\vert \\
&\quad \quad \quad\quad\leq\frac{\alpha(\alpha +1)}{\overline C ^{\beta-1}x^{\alpha +2-\alpha \beta}}+\beta \left(\frac{\alpha ^{2}}{\overline C^{\beta -1}x^{\alpha +2-\alpha \beta}}+\frac{\alpha^{2}}{\overline C ^{2\beta -2}x^{2\alpha +2-2\alpha\beta}}\right)
\\
&\quad \quad \quad\quad\leq\frac{\alpha(\alpha +1)}{\overline C ^{\beta-1}x_0^{\alpha +2-\alpha \beta}}+\beta \left(\frac{\alpha ^{2}}{\overline C^{\beta -1}x_0^{\alpha +2-\alpha \beta}}+\frac{\alpha^{2}}{\overline C ^{2\beta -2}x_0^{2\alpha +2-2\alpha\beta}}\right),
\end{align*}
since both $\alpha +2-\alpha\beta$ and $2\alpha+2-2\alpha\beta$ are positive thanks to $\beta<1+\frac{1}{\alpha}$. Now in view of \eqref{Cbargrand}, we get
$$
\vert g(x)\vert +\beta \vert h(x)\vert\leq \frac{\alpha(\alpha +1)}{x_0^{2}}+\beta\left(\frac{\alpha^{2}}{x_0^2}+\frac{\alpha ^2}{x_0^2}\right)\leq \frac \ep 2,
$$
in virtue of \eqref{xzerogrand}. It therefore follows from \eqref{qqch}  that, for any $(t,x)\in \Omega$ such that $w(t,x)<1$,
$$
\partial _t w(t,x)-\partial _{xx}w(t,x)-f(w(t,x))
\geq 0,
$$
which proves our claim that $\psi$ is a supersolution of \eqref{eq} in $\Omega$.

Let us now have a look at the boundary $\partial \Omega=\{0\}\times[x_0,+\infty) \cup (0,\infty)\times\{x_0\}$.
For $t=0$, $x\geq x_0$, we have $w(0,x)=u_0(x)=u(0,x)$, whereas for $t>0$, $x=x_0$ we have $w(t,x_0)=\frac{1}{(1-\rho(\beta-1)t)^{\frac{1}{\beta-1}}}\geq 1\geq u(t,x_0)$.
Hence $\psi(t,x)\geq u(t,x)$ for any $(t,x)\in\partial \Omega$. We deduce from the comparison principle that
\begin{equation}\label{comparison2}
u(t,x)\leq\psi(t,x)\leq w(t,x),\quad \forall (t,x)\in [0,\infty)\times [x_0,+\infty).
\end{equation}

Now, we can define $T_{\lambda,\frac r 2}\geq t_\lambda$ as in the conclusion \eqref{lowerbound} of Theorem \ref{th:acc}.
For $t\geq T_{\lambda,\frac r 2}$, let us pick a $x\in E_\lambda (t)$. We know from \eqref{lowerbound} that
$x\geq \left(\frac r 2C^{\beta -1} (\beta-1)t\right)^{\frac{1}{\alpha(\beta -1)}}\to +\infty$ as $t\to \infty$ so, up to
 enlarging $T_{\lambda, \frac r 2}$, we can assume that $x\geq x_0$. It therefore follows from \eqref{comparison2} that
$w(t,x)\geq \lambda$ which, using the expression for $w(t,x)$ transfers into
\begin{align*}
x\leq& \left( (\frac{\overline C}{\lambda})^{\beta-1}+(\overline r +\frac
\ep 2)\overline C^{\beta -1}(\beta -1)t\right)^{\frac{1}{\alpha(\beta-1)}}\\
<& \left((\overline r +
\ep )\overline C^{\beta -1}(\beta -1)t\right)^{\frac{1}{\alpha(\beta-1)}}=:x^+(t),
\end{align*}
for $t\geq \overline T_{\lambda,\ep}$, with $\overline T _{\lambda,\ep}\geq T_{\lambda, \frac r 2}$ chosen sufficiently large.
 This proves the upper bound in \eqref{levelset} and therefore concludes the proof of Theorem \ref{th:sandwich}. \qed

\bigskip

\noindent \textbf{Acknowledgement.} The author thanks J.
Coville (for suggesting the issue under consideration), A.
Ducrot (for pointing reference \cite{Zla-05}),  G. Faye and J.-M. Roquejoffre (for pointing reference \cite{Gui-Hua-preprint}), and Q. Griette (for making Figure \ref{hyperbola}).

\bibliographystyle{siam}    
\bibliography{biblio}

\end{document}